\documentclass[12pt]{article}

\usepackage{amsmath,amssymb,amsfonts,amsthm}
\usepackage[colorlinks]{hyperref}
\usepackage{graphicx}
\usepackage{authblk}
\usepackage{longtable}
\usepackage[a4paper, total={6in, 8in}]{geometry}
\usepackage{xcolor}
\usepackage{changes}

\graphicspath{{./figures/}}

\theoremstyle{definition}
\newtheorem{definition}{Definition}

\theoremstyle{theorem}
\newtheorem{theorem}{Theorem}

\theoremstyle{definition}
\newtheorem{example}{Example}

\theoremstyle{remark}
\newtheorem*{remark}{Remark}

\theoremstyle{theorem}
\newtheorem{lemma}{Lemma}

\theoremstyle{definition}

\makeatletter
\AddToHook{cmd/added/before}{\def\Changes@AuthorColor{blue}}
\AddToHook{cmd/deleted/before}{\def\Changes@AuthorColor{red}}
\AddToHook{cmd/replaced/before}{\def\Changes@AuthorColor{orange}}
\makeatother

\title{A Diophantine generalization of Fourier's irrationality trick}

\author{Archan Chattopadhyay and Chandan Chattopadhyay}

\date{}

\allowdisplaybreaks

\begin{document}

\maketitle

\abstract{We revisit Fourier’s method for proving the irrationality of $e$ and generalize it into a unified Diophantine framework, which applies to a broad class of series with bounded or mildly growing integer numerators, and suitably growing, divisible denominators. Using integer linear forms, we establish a general criterion that recovers classical irrationality results while also generating new ones within a single structural approach.}


\section{Introduction}
The problem of determining the irrationality of mathematical constants has a long and influential history in number theory. The earliest known result is traditionally attributed to the ancient Greek philosopher Hippasus, who in the fifth century BC proved the irrationality of $\sqrt 2$. In 1737, Euler first proved the irrationality of $e$ using continued fractions \cite{lang}. A similar approach was undertaken by Lambert in 1761 to produce the first rigorous proof of the irrationality of $\pi$. Fourier is usually credited with providing the first proof of irrationality of a constant defined by series. His classical 1815 proof of the irrationality of $e$ is particularly elegant and well-known, which we briefly review.

The late nineteenth and the last centuries witnessed a shift of focus from isolated constants to universal theorems that prove the irrationality of entire classes of numbers at once. Contributions by Dirichlet, Lindemann, Weierstrass, Cantor, Hilbert, Erd\H{o}s, and Ap\'ery, among others, shaped our modern understanding of the irrationality of numbers. Their efforts culminated in major (albeit technical) theorems in transcendence theory \cite{baker}.

In this article, we take on a more elementary, unified approach to irrationality. Instead of dealing with scattered, ad hoc proofs, we strive to understand the basic principles which distinguish irrationals from rationals. Our viewpoint on irrationality is not a pessimistic property of numbers---as it is often perceived by many---but a structural result of adding up infinitely many rationals. This leads naturally to the study of the irrationality of limits of convergent series. We show that under a set of sufficient conditions, these limits must be irrational. Our approach constitutes a framework which not only unifies standalone proofs, but also generates new irrational numbers. The class of series encompassed by our framework is exceptionally wide, but by no means exhaustive. There remains much to be known and proved. Since this unified framework builds directly upon Fourier's classical proof, it serves as an accessible enrichment topic for advanced high school calculus or introductory undergraduate number theory courses.

The main theorem (Theorem \ref{thm:main_theorem}) and its proof are based on the notion of integer linear forms, and a general irrationality criterion (Section \ref{sec:detour}). Our method proceeds by constructing these integer linear forms from partial sums of the series (Section \ref{sec:ilf}) and showing that they decay to zero (Section \ref{sec:decay}) without vanishing (Section \ref{sec:completion}). Section \ref{sec:applications} highlights the effectiveness of the framework by testing a few special series. Finally, Section \ref{sec:perspectives} offers directions and motivation for budding number theorists.

\subsection{A recap of Fourier's classical proof}
\begin{theorem}[Fourier]
$e$ is irrational.
\end{theorem}
\begin{proof}
Consider the following ubiquitous series
\begin{equation}\label{eq:e}
e=1+1+\dfrac{1}{2!}+\dfrac{1}{3!}+\cdots=\sum_{k=0}^\infty\dfrac{1}{k!}.
\end{equation}
Assume towards a contradiction that $e$ is rational, that is, $e=\dfrac{p}{q}$ for some integers $p,q$. Since the above series is obviously positive, we may take $p,q>0$. Split the right-hand side of \eqref{eq:e} into two parts:
\begin{equation}\label{eq:p_q_split}
\dfrac{p}{q}=\sum_{k=0}^q\dfrac{1}{k!}+\sum_{k=q+1}^\infty\dfrac{1}{k!}.
\end{equation}
Multiply both sides by $q!$ to get
\begin{equation}\label{eq:q!e}
(q-1)!p=\sum_{k=0}^q\dfrac{q!}{k!}+\sum_{k=q+1}^\infty\dfrac{q!}{k!}.
\end{equation}
Realize that $(q-1)!p$ is an integer, and so is the first sum on the right. Hence, the second sum on the right ought to be an integer. However, notice that
\begin{equation}\label{eq:1_by_q}
\sum_{k=q+1}^\infty\dfrac{q!}{k!}=\sum_{j=1}^\infty\dfrac{q!}{(q+j)!}<\sum_{j=1}^\infty\dfrac{1}{(q+1)^j}=\dfrac{1}{q}.
\end{equation}
This dictates that the second sum on the right of \eqref{eq:q!e} lies strictly between $0$ and $1$, and thus cannot be an integer. Hence, $e$ cannot be rational.
\end{proof}

\subsection{More modern results}\label{sec:modern}
Following Fourier's approach, several generalizations have been obtained in more recent mathematics. Consider, for instance, the \emph{Cantor series} \cite{cantor} of a positive real number $x$,
\begin{equation*}
x=\dfrac{a_1}{c_1}+\dfrac{a_2}{c_1 c_2}+\dfrac{a_3}{c_1 c_2 c_3}+\cdots=\sum_{k=1}^\infty a_k\prod_{j=1}^k\dfrac{1}{c_j},
\end{equation*}
where $c_k\geq 2$ and $0\leq a_k<c_k$ are integers. The irrationality of $x$ is determined by \emph{Cantor's theorem}.
\begin{theorem}[Cantor]
$x$ is rational if and only if there exists a natural number $k^*$ such that for all $k\geq k^*$, $a_k=0$ or $a_k=c_k-1$.
\end{theorem}

Subsequent developments include the \emph{Engel series} \cite{engel} of a positive real number $x$,
\begin{equation}\label{eq:engel}
x=\dfrac{1}{c_1}+\dfrac{1}{c_1 c_2}+\dfrac{1}{c_1 c_2 c_3}+\cdots=\sum_{k=1}^\infty\prod_{j=1}^k\dfrac{1}{c_j},
\end{equation}
where the $c_k$ form a nondecreasing sequence of natural numbers and are algorithmically generated from $x$. This expansion is unique; it terminates for rational numbers and is perpetual for irrational numbers. A quick glance at \eqref{eq:e} and \eqref{eq:engel} reveals that $e$ has the Engel expansion $1,1,2,3,4,5,6,7,8,\dots$. An alternating-series version,
\begin{equation*}
x=\dfrac{1}{c_1}-\dfrac{1}{c_1 c_2}+\dfrac{1}{c_1 c_2 c_3}-\cdots=\sum_{k=1}^\infty (-1)^{k-1}\prod_{j=1}^k\dfrac{1}{c_j},
\end{equation*}
called the \emph{Pierce series} \cite{pierce}, applies to real numbers $x$ lying strictly between $0$ and $1$.

\section{A Diophantine detour}\label{sec:detour}
In number theory, \emph{Diophantine approximation} involves the study of how well real numbers can be ``approximated" by rational numbers. Sufficiently good rational approximations (or lack thereof) often imply that a number is irrational (or rational). To this extent, we first define the notion of a linear form.
\begin{definition}
A \emph{linear form} in $x$ and $1$ (or simply $x$) is an expression of the form
\begin{equation*}
\alpha\cdot x+\beta\cdot 1,
\end{equation*}
where $\alpha$ and $\beta$ are real numbers. If they are integers, the resulting expression is called an \emph{integer linear form} in $x$.
\end{definition}
Now we state a general irrationality criterion (see \cite[Theorem 1.5]{duverney}) which we shall use in Section \ref{sec:completion}.
\begin{theorem}\label{thm:irrationality_criterion}
Let $x\in\mathbb{R}$. Suppose there exist infinite sequences of integers $(\alpha_m)$ and $(\beta_m)$, with $\alpha_m\neq 0$ for all $m$, such that for the integer linear form $\lambda_m:=\alpha_m x-\beta_m$,
\begin{equation}\label{eq:irrationality_theorem}
0<|\lambda_m|\leq\epsilon_m,
\end{equation}
where $\displaystyle{\lim_{m\to\infty}}\epsilon_m=0$. Then, $x$ is irrational.
\end{theorem}
\begin{proof}
Assume $x=\dfrac{p}{q}\in\mathbb{Q}$ is in lowest terms. Then,
\begin{equation*}
|\lambda_m|=\dfrac{|\alpha_m p-\beta_m q|}{|q|}.
\end{equation*}
Hence, for every $m$, either $|\lambda_m|=0$ or
\begin{equation}\label{eq:contradiction}
|\lambda_m|\geq\dfrac{1}{|q|}.
\end{equation}
By \eqref{eq:irrationality_theorem}, $|\lambda_m|\leq\epsilon_m\to 0$ as $m\to\infty$. So, for all sufficiently large $m$, we have $|\lambda_m|\leq\epsilon_m<\dfrac{1}{|q|}$, contradicting \eqref{eq:contradiction}. Therefore, $x$ cannot be rational.
\end{proof}

\section{Sufficient conditions for irrationality (main theorem)}\label{sec:constraints}
Now we come to the heart of this article and state the main theorem.
\begin{theorem}\label{thm:main_theorem}
Let $X$ (the constant under consideration) be defined by the series
\begin{equation}\label{eq:C_n}
X:=\sum_{k=0}^\infty\dfrac{a_k}{b_k},
\end{equation}
where the $a_k$ and $b_k$ are subject to the following constraints:
\begin{enumerate}
\item $a_k\in\mathbb{Z}\backslash\lbrace 0\rbrace$ and $b_k\in\mathbb{N}$.
\item $|a_k|\leq a^*$ for all $k$ and some $a^*\in\mathbb{N}$ (in other words, $a_k=O(1)$).
\item $b_k\mid b_{k+1}$ and $b_k<b_{k+1}$.
\item $\displaystyle{\lim_{k\to\infty}}\dfrac{b_k}{b_{k+1}}=0$ (in other words, $\dfrac{b_k}{b_{k+1}}=o(1)$).
\end{enumerate}
Then, $X$ converges absolutely and is irrational.
\end{theorem}

\begin{remark}\label{cor:main_theorem}
All but the first of the above hypotheses need only hold for all sufficiently large $k$, since discarding finitely many terms only changes $X$ by a rational number. If $a_k=0$ for some $k$, the corresponding terms may be deleted; provided infinitely many nonzero terms remain, the above hypotheses hold along the surviving subsequence. Indeed, if we denote the surviving denominators by $b_{k_j}$, then $b_{k_j}\mid b_{{k_j+1}}\mid b_{k_{j+1}}$, $b_{k_j}<b_{{k_j+1}}\leq b_{k_{j+1}}$ and $\dfrac{b_{k_j}}{b_{k_{j+1}}}\leq\dfrac{b_{k_j}}{b_{{k_j}+1}}=o(1)$.
\end{remark}

Observe how each hypothesis applies to and generalizes \eqref{eq:e}. (For more information on the $O$- and $o$-notation, refer to \cite[Appendix A.2]{flajolet}.) Absolute convergence \cite[Section 10.18]{apostol_calculus} follows from a direct comparison test. By the third hypothesis, $\dfrac{b_{k+1}}{b_k}\geq 2$, implying $b_k\geq 2^k$. Therefore, by the second hypothesis, $\left|\dfrac{a_k}{b_k}\right|\leq\dfrac{a^*}{2^k}$, and $\displaystyle{\sum_{k=0}^\infty}\left|\dfrac{a_k}{b_k}\right|\leq a^*\displaystyle{\sum_{k=0}^\infty}\left(\dfrac{1}{2}\right)^k=2a^*$.

We shall now commence the proof of the irrationality of $X$.

\section{Construction of integer linear forms}\label{sec:ilf}
Let $m\in\mathbb{N}_0$ (which is the set of whole numbers including zero). Split $X$ into two parts:
\begin{equation}\label{eq:C_n_split}
X=X^{(\leq m)}+X^{(>m)},
\end{equation}
where
\begin{equation}\label{eq:split_halves}
X^{(\leq m)}:=\sum_{k=0}^m\dfrac{a_k}{b_k},\quad X^{(>m)}:=\sum_{k=m+1}^\infty\dfrac{a_k}{b_k}.
\end{equation}
This step is reminiscent of \eqref{eq:p_q_split}.
\begin{lemma}
With
\begin{equation}\label{eq:A_m_B_m}
A_m:=b_m\in\mathbb{N},\quad B_m:=b_m X^{(\leq m)}\in\mathbb{Z},
\end{equation}
we have the integer linear form
\begin{equation}\label{eq:integer_linear_form}
\Lambda_m:=A_m X-B_m=b_m X^{(>m)}.
\end{equation}
\end{lemma}
\begin{proof}
Multiply \eqref{eq:C_n} by $b_m$ and use \eqref{eq:C_n_split}. The choice \eqref{eq:A_m_B_m} guarantees every $B_m$ is an integer (by the third hypothesis in Theorem \ref{thm:main_theorem}, $b_m$ clears all denominators $b_k$ for $k\leq m$).
\end{proof}
The multiplication by $b_m$ reminds us of \eqref{eq:q!e}.

\section{Decay of the integer linear forms}\label{sec:decay}
From \eqref{eq:split_halves} and \eqref{eq:integer_linear_form}, using the triangle inequality, we have
\begin{equation*}
|\Lambda_m|\leq b_m\sum_{k=m+1}^\infty\dfrac{|a_k|}{b_k}.
\end{equation*}
Setting $j=k-m$, we get
\begin{equation}\label{eq:Lambda_m_decay}
|\Lambda_m|\leq\sum_{j=1}^\infty |a_{m+j}|\dfrac{b_m}{b_{m+j}}.
\end{equation}
We now come to a crucial lemma.
\begin{lemma}\label{lem:f(m)}
There exists a nondecreasing and unbounded function $f:\mathbb{N}_0\to\mathbb{N}\backslash\lbrace 1\rbrace$ such that
\begin{equation}
\dfrac{b_m}{b_{m+j}}\leq\dfrac{1}{f(m)^j},\quad j\geq 1.
\end{equation}
\end{lemma}
\begin{proof}
By the third and fourth hypotheses in Theorem \ref{thm:main_theorem}, there exists a sequence
\begin{equation*}
g(k):=\dfrac{b_{k+1}}{b_k}\in\mathbb{N}\backslash\lbrace 1\rbrace,
\end{equation*}
with $\displaystyle{\lim_{k\to\infty}}g(k)=\infty$. Define
\begin{equation*}
f(k):=\min_{j\geq k} g(j).
\end{equation*}
Then, $f(k)\in\mathbb{N}\backslash\lbrace 1\rbrace$ is nondecreasing and unbounded. Moreover, $f(k)\leq g(k)$, so
\begin{equation*}
\dfrac{b_k}{b_{k+1}}\leq\dfrac{1}{f(k)}.
\end{equation*}
This yields
\begin{equation*}
\dfrac{b_m}{b_{m+j}}=\prod_{i=1}^j\dfrac{b_{m+i-1}}{b_{m+i}}\leq\prod_{i=1}^j\dfrac{1}{f(m+i-1)}\leq\dfrac{1}{f(m)^j}.
\end{equation*}
\end{proof}
By the second hypothesis in Theorem \ref{thm:main_theorem}, $|a_{m+j}|\leq a^*$. Hence, from \eqref{eq:Lambda_m_decay},
\begin{equation}\label{eq:Lambda_m_f_m}
|\Lambda_m|\leq a^*\sum_{j=1}^\infty\dfrac{1}{f(m)^j}=\dfrac{a^*}{f(m)-1}.
\end{equation}
The above inequality is in the spirit of \eqref{eq:1_by_q}. With this, we arrive at the following result.
\begin{lemma}
With
\begin{equation}\label{eq:delta_m}
\delta_m:=\dfrac{a^*}{f(m)-1},
\end{equation}
we get
\begin{equation}\label{eq:Lambda_m_delta_m}
|\Lambda_m|\leq\delta_m,\quad\lim_{m\to\infty}\delta_m=0.
\end{equation}
\end{lemma}
\begin{proof}
The result follows immediately from \eqref{eq:delta_m} since $f$ is unbounded.
\end{proof}

\section{Completion of the proof}\label{sec:completion}
Before claiming that $X$ is irrational, we must cater to the following essential nonvanishing requirement. Without this condition, the irrationality criterion would fail, since vanishing linear forms may occur even for rational numbers.
\begin{lemma}\label{lem:nonvanishing}
Under the hypotheses of Theorem \ref{thm:main_theorem}, $\Lambda_m\neq 0$ for all sufficiently large $m$.
\end{lemma}
\begin{proof}
From \eqref{eq:split_halves} and \eqref{eq:integer_linear_form}, we have
\begin{equation*}
\Lambda_m=\dfrac{b_m}{b_{m+1}}\left(a_{m+1}+\sum_{j=2}^\infty a_{m+j}\dfrac{b_{m+1}}{b_{m+j}}\right).
\end{equation*}
Since $a_{m+1}$ is a nonzero integer, $|a_{m+1}|\geq 1$. To show $\Lambda_m\neq 0$, we must show that the term in the parentheses is nonzero. Let
\begin{equation*}
S_m:=\sum_{j=2}^\infty a_{m+j}\dfrac{b_{m+1}}{b_{m+j}}.
\end{equation*}
In direct analogy with \eqref{eq:Lambda_m_f_m}, we obtain
\begin{equation*}
|S_m|\leq\dfrac{a^*}{f(m+1)-1}.
\end{equation*}
Since $f$ grows unbounded, there exists a threshold $m^*$ such that for all $m\geq m^*$, $|S_m|<1$. For those $m$, the reverse triangle inequality yields
\begin{equation*}
|a_{m+1}+S_m|\geq |a_{m+1}|-|S_m|>1-1=0.
\end{equation*}
This shows that $\Lambda_m\neq 0$ for all sufficiently large $m$.
\end{proof}

Finally, from \eqref{eq:integer_linear_form}, \eqref{eq:Lambda_m_delta_m} and the above lemma, we see that $X$ meets all the conditions in Theorem \ref{thm:irrationality_criterion}. We thus conclude that $X$ is irrational. This completes the proof of Theorem \ref{thm:main_theorem}.

\section{Applications}\label{sec:applications}
We illustrate the utility and versatility of the developed framework by selecting a few explicit series. Of course, basic numbers such as $\sin 1$, $\cos 1$, $\sinh 1$ and $\cosh 1$ may be easily proved irrational using our method (see \cite{osler,spiegel} for similar proof techniques). Here, however, we shall dive into more advanced examples.

\begin{example}\label{ex:double_factorial}
Consider the series
\begin{equation}\label{eq:double_factorial}
1+\dfrac{1}{3}+\dfrac{1}{3\cdot 5}+\dfrac{1}{3\cdot 5\cdot 7}+\cdots=\sum_{k=0}^\infty\dfrac{1}{(2k+1)!!}.
\end{equation}
In this case, we have $a_k\equiv 1$, $b_k=(2k+1)!!$. Clearly, $b_k\mid b_{k+1}$, and $\dfrac{b_k}{b_{k+1}}=\dfrac{1}{2k+3}=o(1)$. Hence, this constant is irrational. In fact, it has the following closed-form expression:
\begin{equation*}
\sum_{k=0}^\infty\dfrac{1}{(2k+1)!!}=\sqrt{\dfrac{e\pi}{2}}\operatorname{erf}\left(\dfrac{1}{\sqrt{2}}\right),
\end{equation*}
where $\operatorname{erf}$ is the error function \cite[(7.1.1)]{abramowitz}. Although, in this case, irrationality may also be proved via the Engel series \eqref{eq:engel}, our framework does so in a direct and systematic (to put it another way: ``plug-and-play") manner.

Figure \ref{fig:double_factorial} shows a semilog plot of $|\Lambda_m|$ vs.\ $m$ (for $m\leq 1000$) for the series \eqref{eq:double_factorial}, in accordance with \eqref{eq:integer_linear_form}. Observe how the magnitude of the integer linear form $\Lambda_m$ steadily decays towards zero but does not vanish, illustrating the mechanism underlying the irrationality proof.

\begin{figure}[h]
\centering
\includegraphics[width=0.6\textwidth]{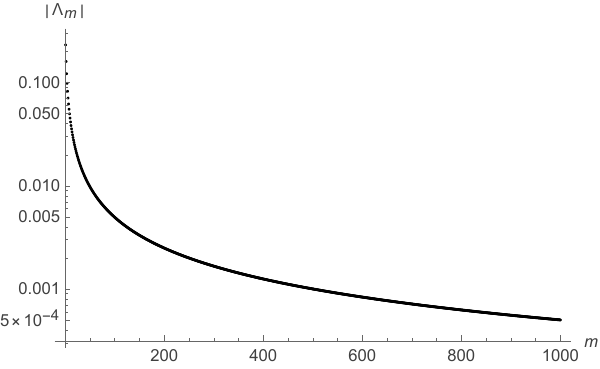}
\caption{Plot of $|\Lambda_m|$ vs.\ $m$ for \eqref{eq:double_factorial}.}
\label{fig:double_factorial}
\end{figure}
\end{example}

\begin{example}
Let $p_k$ be the $k$th prime. The \emph{primorial} of $p_k$ is defined as the product of all primes up to $p_k$, and is denoted as
\begin{equation*}
p_k\#:=\prod_{j=1}^k p_j.
\end{equation*}
Consider the sum of the reciprocals of the primorials
\begin{equation*}
\dfrac{1}{2}+\dfrac{1}{2\cdot 3}+\dfrac{1}{2\cdot 3\cdot 5}+\dfrac{1}{2\cdot 3\cdot 5\cdot 7}+\dfrac{1}{2\cdot 3\cdot 5\cdot 7\cdot 11}+\cdots=\sum_{k=1}^\infty\dfrac{1}{p_k\#}.
\end{equation*}
Here, we have $b_k=p_k\#$, and $\dfrac{b_k}{b_{k+1}}=\dfrac{1}{p_{k+1}}$. Since $\displaystyle{\lim_{k\to\infty}}p_{k+1}=\infty$, we may write $\dfrac{b_k}{b_{k+1}}=o(1)$. Hence, this prime-related constant is irrational. (For an alternative proof based on Cantor series, see \cite{griffiths}.)
\end{example}

\begin{example}\label{ex:favourite}
As a matter of fact, the $b_k$ may be any positive integer sequence which grows unbounded, whilst maintaining the divisibility chain and supergeometric growth (the third and fourth hypotheses in Theorem \ref{thm:main_theorem}, respectively). Thus, for instance, one may choose the $b_k$ as $(k!)^n$ ($n\in\mathbb{N}$), $2^{2k} (k!)^2$, the cumulative product of the central binomial coefficients $\displaystyle{\binom{2k}{k}}$, the cumulative product of the Fibonacci sequence (starting from $1$ and $2$), or one's personal favourite. While the first three hypotheses are trivial to verify, let us explicitly test these sequences for $b_k$ against the fourth hypothesis:
\begin{itemize}
\item For $b_k=(k!)^n$, $\dfrac{b_k}{b_{k+1}}=\dfrac{1}{(k+1)^n}=o(1)$.
\item For $b_k=2^{2k}(k!)^2$, $\dfrac{b_k}{b_{k+1}}=\dfrac{1}{4(k+1)^2}=o(1)$.
\item If $b_k$ is the cumulative product of the central binomial coefficients $\displaystyle{\binom{2k}{k}}$, then $\dfrac{b_k}{b_{k+1}}=\dfrac{1}{\displaystyle{\binom{2k+2}{k+1}}}=o(1)$.
\item If $b_k$ is the cumulative product of the Fibonacci numbers $F_k$ (with seed values $1$ and $2$), then $\dfrac{b_k}{b_{k+1}}=\dfrac{1}{F_{k+1}}=o(1)$.
\end{itemize}
\end{example}

\begin{example}\label{ex:mu}
A striking aspect of our framework is the complete freedom it provides to the signs of the $a_k$---a feat which arises due to the second hypothesis of Theorem \ref{thm:main_theorem}---as opposed to the methods discussed in Section \ref{sec:modern}. This means that given any series which abides by the hypotheses of Theorem \ref{thm:main_theorem}, we may alter the signs of any of the $a_k$ at will and still acquire irrationality. To this extent, consider the series
\begin{equation*}
\sum_{k=0}^\infty\dfrac{\mu(k)}{k!},
\end{equation*}
where $\mu$ is the M\"obius function \cite[Section 2.2]{apostol}, which assumes values $-1,0,1$. A quick inspection reveals that the hypotheses in Theorem \ref{thm:main_theorem} are obeyed after restricting to the subsequence for which $\mu(k)\neq 0$ (that is, by the remark in Section \ref{sec:constraints}). Hence, this series has an irrational limit.
\end{example}

\begin{example}
In fact, we may use for the $a_k$ any bounded integer sequence or number-theoretic function, such as the Liouville function $\lambda$, the magnitude of the Dirichlet character $\chi$, and the prime indicator function (for definitions, consult \cite{apostol}). A particularly curious case is
\begin{equation*}
\sum_{k=0}^\infty\dfrac{k\bmod n}{k!},\quad n\in\mathbb{N}\backslash\lbrace 1\rbrace,
\end{equation*}
where $k\bmod n$ is the residue modulo $n$ \cite[Section 4.2]{burton}, the set of residues taking values $0,1,\dots,n-1$. The remark in Section \ref{sec:constraints} guarantees that this series is irrational.
\end{example}

\begin{example}
Any rational number $x$ with a nonterminating decimal expansion may be represented as
\begin{equation*}
x=\sum_{k=0}^\infty\dfrac{d_k}{10^k},
\end{equation*}
where $d_0$ is the integer part and $d_k$ ($k\geq 1$) is the digit at the $k$th decimal place. Consider the following ``stretched" version of $x$:
\begin{equation}\label{eq:x_stretched}
x^*:=\sum_{k=0}^\infty\dfrac{d_k}{10^{k^2}}.
\end{equation}
For instance, if $x=\dfrac{1}{7}=0.\overline{142857}$, then
\begin{equation*}
x^*=0.100400002000000800000000500000000007\dots.
\end{equation*}
A quick tally with the hypotheses in Theorem \ref{thm:main_theorem} (and the remark in Section \ref{sec:constraints}) shows that $x^*$ is irrational. The plot in Figure \ref{fig:digits} apparently verifies this: observe the recurring pattern in the digits of $x$, and the dearth of a pattern in those of $x^*$.

\begin{figure}[h]
\centering
\includegraphics[width=0.6\textwidth]{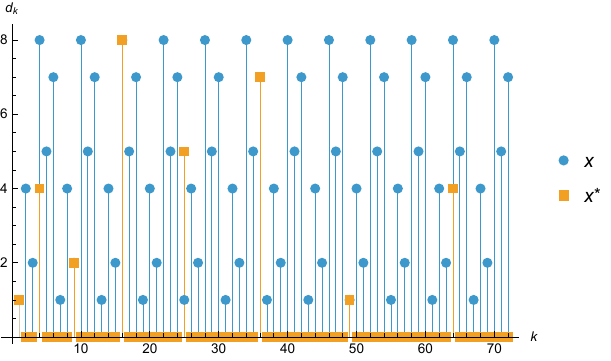}
\caption{Plot of $d_k$ vs.\ $k$ for $x=\dfrac{1}{7}$ (blue) and $x^*$ defined by \eqref{eq:x_stretched} (orange).}
\label{fig:digits}
\end{figure}
\end{example}

\begin{remark}
While Theorem \ref{thm:irrationality_criterion} is a standard criterion in Diophantine approximation, the main result of Theorem \ref{thm:main_theorem} is, to the best of our knowledge, novel. The hypotheses of the theorem encompass a wider class of constants than the methods discussed in Section \ref{sec:modern}. Perhaps the most conspicuous advantage of this framework is its agnosticism to the signs of the numerators $a_k$. Cantor and Engel series, on the other hand, are restricted to positive numerators, whereas Pierce series can only accommodate alternating signs. Of course, our framework generalizes two of these methods: an Engel series is obtained by setting $a_k\equiv 1$ with denominators $c_k=\dfrac{b_k}{b_{k-1}}$, and a Pierce series is obtained with $a_k=(-1)^k$. Whereas Cantor series restrict the range of the individual $a_k$ (without mandating a global bound), Theorem \ref{thm:main_theorem} only requires the $a_k$ to be globally bounded in magnitude. The next example substantiates this disparity between the two approaches.
\end{remark}

\begin{example}\label{ex:disparity}
In this example, we address the incomparability of Cantor series and Theorem \ref{thm:main_theorem}. Consider a Cantor series where $a_k=k$ and $c_k=k+2$:
\begin{equation*}
\dfrac{1}{3}+\dfrac{2}{3\cdot 4}+\dfrac{3}{3\cdot 4\cdot 5}+\dfrac{4}{3\cdot 4\cdot 5\cdot 6}+\cdots=\sum_{k=1}^\infty k\prod_{j=1}^k\dfrac{1}{j+2}.
\end{equation*}
By Cantor's theorem, this series is irrational. However, it clearly violates the second hypothesis of Theorem \ref{thm:main_theorem}. On the other hand, consider the following series with $a_k=(-1)^{k-1}$ and $b_k=\displaystyle{\prod_{j=1}^k} c_j$:
\begin{equation*}
\dfrac{1}{3}-\dfrac{1}{3\cdot 4}+\dfrac{1}{3\cdot 4\cdot 5}-\dfrac{1}{3\cdot 4\cdot 5\cdot 6}+\cdots=\sum_{k=1}^\infty (-1)^{k-1}\prod_{j=1}^k\dfrac{1}{j+2}.
\end{equation*}
In this case, all the hypotheses of Theorem \ref{thm:main_theorem} are satisfied, but it is not a valid Cantor series (owing to the negative signs).
\end{example}

\section{Further perspectives}\label{sec:perspectives}
The hypotheses listed in Theorem \ref{thm:main_theorem} specifically cater to numerators $a_k$ bounded in magnitude, and denominators $b_k$ exhibiting supergeometric growth. Therefore, any sufficiently ``fast-converging" series is guaranteed to be irrational. On the contrary, it is usually much harder to prove the irrationality of rather ``slow-converging" series (that is, series whose denominators increase subgeometrically). This includes, for instance, \emph{Dirichlet series} \cite[Chapter 11]{apostol}
\begin{equation}\label{eq:dirichlet_series}
\sum_{k=1}^\infty\dfrac{a_k}{k^s},\quad s\in\mathbb{N},
\end{equation}
whose denominators grow polynomially. Figure \ref{fig:denominator_growth_comparison} shows a semilog plot which visually compares different denominator growth rates. Notice how factorial growth eventually outpaces polynomial (the slowest) and geometric (exponential) growths. The region above the orange curve is the regime in which our irrationality criterion applies.

\begin{figure}[h]
\centering
\includegraphics[width=0.6\textwidth]{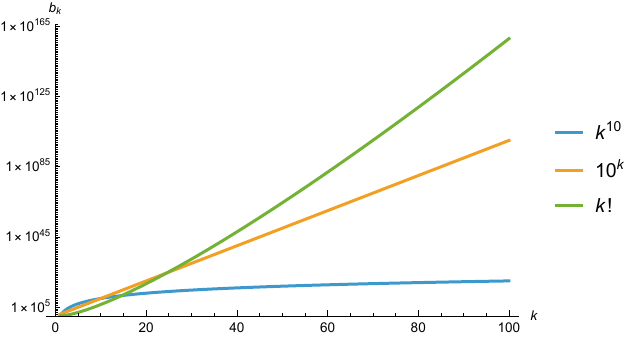}
\caption{Plot comparing different denominator growth rates. Note that the third and fourth hypotheses of Theorem \ref{thm:main_theorem} are explicitly satisfied in the region above the orange curve, while the fourth hypothesis is violated on and below it.}
\label{fig:denominator_growth_comparison}
\end{figure}

For $a_k\equiv 1$, \eqref{eq:dirichlet_series} yields the well-known \emph{Riemann zeta function}
\begin{equation*}
\zeta(s)=1+\dfrac{1}{2^s}+\dfrac{1}{3^s}+\cdots=\sum_{k=1}^\infty\dfrac{1}{k^s},\quad s\in\mathbb{N}\backslash\lbrace 1\rbrace,
\end{equation*}
whose irrationality is well-understood for even $s$. (Euler worked out the celebrated formula \cite[Section 1.5]{edwards}
\begin{equation*}
\zeta(2n)=\dfrac{(-1)^{n+1}B_{2n}(2\pi)^{2n}}{2(2n)!},\quad n\in\mathbb{N}
\end{equation*}
(where $B_n$ are the Bernoulli numbers), whereas it took Lambert, Legendre and Lindemann \cite{lindemann} to prove the irrationality of integral powers of $\pi$.) For odd $s$, however, only the case of $s=3$ has been settled (by Ap\'ery \cite{apery,poorten} and Beukers \cite{beukers}), and the question for $s\geq 5$ remains open (although it is known that infinitely many odd zeta values are irrational \cite{rivoal}, and that at least one of $\zeta(5),\zeta(7),\zeta(9),\zeta(11)$ is irrational \cite{zudilin}). Amidst this disappointment, one may find solace in knowing that the factorial analogue (or ``cousin") of the zeta function,
\begin{equation*}
\sum_{k=1}^\infty\dfrac{1}{(k!)^s},
\end{equation*}
is irrational (as mentioned in Example \ref{ex:favourite}). This series is actually a special case of the \emph{generalized hypergeometric series} \cite{bailey} $_0F_{s-1}$. One may alternatively try to prove the irrationality of the reciprocal of the zeta function
\begin{equation*}
\dfrac{1}{\zeta(s)}=\sum_{k=1}^\infty\dfrac{\mu(k)}{k^s},
\end{equation*}
whose factorial version (for $s=1$) is settled in Example \ref{ex:mu}, and may be easily generalized for arbitrary $s$ using our framework.

Another interesting open problem is the irrationality of the \emph{Dirichlet beta function}
\begin{equation*}
\beta(s)=1-\dfrac{1}{3^s}+\dfrac{1}{5^s}-\dfrac{1}{7^s}+\cdots=\sum_{k=0}^\infty\dfrac{(-1)^k}{(2k+1)^s},\quad s\in\mathbb{N},
\end{equation*}
for even $s$. A fascinating case of this is $\beta(2)$, often referred to as \emph{Catalan's constant} $G$:
\begin{equation*}
G=1-\dfrac{1}{3^2}+\dfrac{1}{5^2}-\dfrac{1}{7^2}+\cdots=\sum_{k=0}^\infty\dfrac{(-1)^k}{(2k+1)^2}.
\end{equation*}
In the words of Stewart \cite{stewart}, ``There is a rich and seemingly endless source of definite integrals that can be equated to or expressed in terms of Catalan's constant." As a result, it ``continues to remain one of the most inscrutable constants in mathematics where the question concerning its irrationality is not settled." Catalan's constant may be expressed in terms of Dirichlet characters as
\begin{equation*}
G=\sum_{k=1}^\infty\dfrac{\chi_4(k)}{k^2},
\end{equation*}
where $\chi_4$ is the primitive Dirichlet character modulo $4$, defined by
\begin{equation*}
\chi_4(k):=
\begin{cases}
1,&k\equiv 1\pmod 4,\\
-1,&k\equiv 3\pmod 4,\\
0,&k\text{ even}.
\end{cases}
\end{equation*}
Our framework successfully proves the irrationality of the related series
\begin{equation*}
\sum_{k=1}^\infty\dfrac{\chi_4(k)}{(k!)^2}.
\end{equation*}

A natural extension of our work is incorporating the case where the numerators $a_k$ are not necessarily bounded. For instance, a simple upgrade of the second hypothesis in Theorem \ref{thm:main_theorem} would be
\begin{equation*}
a_k=O\left(\left(\dfrac{b_k}{b_{k-1}}\right)^\theta\right),\quad 0\leq\theta <1.
\end{equation*}
Note that we recover the original hypothesis for the case $\theta=0$. On redoing the maths in Section \ref{sec:decay}, one checks that now, \eqref{eq:delta_m} yields
\begin{equation*}
\delta_m=a^*\dfrac{f(m)^\theta}{f(m)-1},
\end{equation*}
so \eqref{eq:Lambda_m_delta_m} and Lemma \ref{lem:nonvanishing} still hold, guaranteeing irrationality. For the special case of factorial denominators, we get $a_k=O(k^\theta)$. In this case, any function with sublinear growth may be substituted in the numerator. Examples include \cite[Section 22.10]{hardy}
\begin{itemize}
\item $\tau$, which counts the number of divisors ($\tau(k)=k^{o(1)}$).
\item $\omega$, which counts the number of prime factors ($\omega(k)=O(\log k)$).
\item $\Omega$, which counts the number of prime factors with multiplicity ($\Omega(k)=O(\log k)$).
\item $\lfloor\sqrt k\rfloor$, $\lfloor\log k\rfloor$.
\end{itemize}
Further (and rather dramatic) relaxation of the second hypothesis could have immediate consequences of utmost importance. As a token of motivation, consider the following expansion of $\pi$ (you may check it using any computer algebra system):
\begin{equation*}
\pi=2\sum_{k=0}^\infty\dfrac{k!}{(2k+1)!!}.
\end{equation*}
Adding \eqref{eq:e} and the above series, we obtain
\begin{equation*}
e+\pi=\sum_{k=0}^\infty\dfrac{(2k+1)!!+2(k!)^2}{k!\,(2k+1)!!}.
\end{equation*}
If the second hypothesis could be generalized enough so as to accommodate the above structure, it would provide potential insight into the longstanding problem of the irrationality of $e+\pi$.

To end on an amusing note, one may artificially create at will a most pathological series and still claim its irrationality. A ``good" example would be
\begin{equation*}
\sum_{k=1}^\infty\dfrac{(-1)^{k-1}\mu(k)\chi_4(k)(k\bmod 2026)}{2^{2k}(k!)^2 p_k\#}.
\end{equation*}

%
%
%

\section*{}
Archan Chattopadhyay\\
Howrah, West Bengal, India\\
\href{mailto:archanc@alum.iisc.ac.in}{archanc@alum.iisc.ac.in}\\
ORCID iD: 0000-0003-4132-4688
\\\\
Chandan Chattopadhyay\\
Department of Mathematics, Narasinha Dutt College\\
Howrah, West Bengal -- 711101, India\\
\href{mailto:chandanmath2011@gmail.com}{chandanmath2011@gmail.com}\\
ORCID iD: 0009-0008-0394-9146

\end{document}